\documentclass[10pt, english]{article}
\usepackage{geometry}   
\usepackage{amssymb}
\usepackage{amsmath}
\usepackage{amsfonts}
\usepackage{latexsym}
\usepackage{graphicx}
\usepackage{amsthm,enumerate,exscale,stmaryrd}
\usepackage{cases}
\usepackage{bm}
\usepackage[pdftex,bookmarks,colorlinks,breaklinks=true]{hyperref}
\usepackage{color}
\input xy
\xyoption{all}
\usepackage{dsfont}
\usepackage{yfonts}
\usepackage{bbm}
\usepackage{leftidx}
\usepackage{upgreek}
\usepackage[all]{xy}
\usepackage{verbatim}
\usepackage[multiple]{footmisc}
\usepackage{charter}
\NoComputerModernTips
\newdir^{ (}{{}*!/-3pt/\dir^{(}}
\newdir_{ (}{{}*!/-3pt/\dir_{(}}

\newcommand{\xqed}[1]{%
  \leavevmode\unskip\penalty9999 \hbox{}\nobreak\hfill
  \quad\hbox{\ensuremath{#1}}}

\def\Sym{\mathop{\rm Sym}\nolimits}

\def\Ens{\mathop{\rm Ens}\nolimits}

\def\Spec{\mathop{\bf Spec}\nolimits}

\def\rank{\mathop{\rm rank}\nolimits}

\let\phi\varphi
\let\epsilon\varepsilon
\let\setminus\smallsetminus

\newtheorem{thm}[equation]{Theorem}
\newtheorem{cor}[equation]{Corollary}

\newtheorem{lem}[equation]{Lemma}

\newtheorem{prop}[equation]{Proposition}

\theoremstyle{definition}

\newtheorem{rem}[equation]{Remark}

\newtheorem{cons}[equation]{Construction}
\newtheorem{notation}[equation]{Notations}

\newcounter{example}
\renewcommand{\theexample}{\arabic{example}}

\numberwithin{equation}{section}

\newcommand{\BG}{{\mathbb{G}}}

\newcommand{\BM}{{\mathbb{M}}}

\newcommand{\BP}{{\mathbb{P}}}
\newcommand{\BQ}{{\mathbb{Q}}}

\newcommand{\BZ}{{\mathbb{Z}}}

\newcommand{\Fp}{{\mathfrak{p}}}

\newcommand{\FA}{{\mathfrak{A}}}

\newcommand{\FS}{{\mathfrak{S}}}
\newcommand{\FT}{{\mathfrak{T}}}

\newcommand{\CC}{{\cal C}}

\newcommand{\CE}{{\cal E}}
\newcommand{\CF}{{\cal F}}
\newcommand{\CG}{{\cal G}}
\newcommand{\CH}{{\cal H}}

\newcommand{\CO}{{\cal O}}

\newcommand{\CV}{{\cal V}}

\newcommand{\ep}{{ \bigwedge}}

\def\into{\hookrightarrow}

\newbox\mybox
\def\arrover#1{\mathrel{
       \setbox\mybox=\hbox spread 1em
              {\hfil$\scriptstyle#1\vphantom{g}$\hfil}
       \vbox{\offinterlineskip\copy\mybox
             \hbox to\wd\mybox{\rightarrowfill}}}}

\def\ontoover#1{\mathrel{
       \setbox\mybox=\hbox spread 1.4em{\hfil$\scriptstyle#1$\hfil}
       \vbox{\offinterlineskip\copy\mybox
             \hbox to\wd\mybox{\rightarrowfill\hskip-2.8mm
                               $\rightarrow$}}}}

\def\to{\rightarrow}

\title{On Certain Morphisms between Flag Varieties}
\author{S. Mohammad Hadi Hedayatzadeh}

\begin{document}
\date{}

\maketitle
\abstract{The aim of this paper is to construct certain closed embeddings of Grassmannian varieties, using tensor operations on vector bundles. These embeddings generalize Segre and Pl\"ucker morphisms.}
%\abstract{}
%\normalsize
\tableofcontents

%\makeatletter
%\@addtoreset{thm}{chapter}
%\makeatother
%\frontmatter
\parindent0pt
%INTRODUCTION

%\addtocounter{section}{-1}

\section{Introduction}

Grassmannians or more generally flag varieties, which are natural generalizations of projective spaces, play a significant role in algebraic geometry. Let $F$ be a field, and $1\leq d\leq n$ be natural numbers. The Grassmannian variety $\BG r(n,d)$ is the space of all $d$-dimensional subspaces of $F^{n}$. In a more canonical way, if $X$ is a scheme and $\CV$ is a vector bundle on $X$, then $\BG r(\CV,d)$ is the scheme that represents the functor
\begin{align*}
\mathfrak{Sch}_{X}&\to \Ens
\end{align*}
sending $S$ to the set of isomorphism classes of all short exact sentences of locally free $\CO_{S}$-modules
\[0\to\CF\to\CV_{S}\to\CG\to0\] where $\CF$ is of rank $d$.\\

The existence or non-existence of morphisms of Grassmannian varieties have been extensively studied (see \cite{MR1016272,MR1274499,MR2520917}). For example, the Segre embedding is the following morphism between projective spaces:
\begin{align*}
\BP^{m}\times\BP^{n}&\to\BP^{(m+1)(n+1)-1}\\
\big([x_{0},\dots,x_{m}],[y_{0},\dots,y_{n}]\big)&\mapsto [x_{i}y_{j}]_{0\leq i\leq m,0\leq j\leq n}
\end{align*}

and the Pl\"ucker morphism is the following morphism between Grassmannian varieties:
\begin{align*}
\BG r(n,d)&\to\BP^{\binom{n}{d}-1}\\
W&\mapsto \ep^{d}W
\end{align*}

In this paper, we use tensor products, symmetric and alternating powers of vector bundle on schemes to define natural morphisms between Grassmannian varieties. More precisely, we construct the following morphisms:

\begin{align*}
\FT:\BG r(\CV_{1},m_{1})\times_{X}\dots\times_{X}\BG r(\CV_{r},m_{r})&\to \BG r(\CV_{1}\otimes_{\CO_{X}}\dots\otimes_{\CO_{X}}\CV_{r},m_{1}\cdots m_{r})\\
\FT_{r}:\BG r(\CV,m)&\to \BG r\big(\CV^{\otimes r},m^{r}\big)\\
\FA_{r}:\BG r(\CV,m)&\to \BG r\Big(\ep^{r}_{\CO_{X}}\CV,\binom{m}{r}\Big)\\
\FS_{r}:\BG r(\CV,m)&\to \BG r\Big(\Sym^{r}_{\CO_{X}}\CV,\binom{m+r-1}{r}\Big)
\end{align*}

where $\CV_{1},\dots,\CV_{r}$ and $\CV$ are finite locally free sheaves on a scheme $X$. The main result of the paper is then the following theorem:\\

\textbf{Theorem \ref{MainThm}}. Morphisms $\FT, \FT_{r}$ and $\FA_{r}$ are closed immersions. Morphism $\FS_{r}$ is a closed immersion, if $m$ is at least $2$ or $r$ is invertible on $X$.\\

We should also mention that it is possible to generalize these results to more general flag varieties. These general results can be obtained by induction from what we already have. There are many ways of constructing flags from a given flag and using various tensor constructions. Therefore, even the statement of a general result would be very complex and so, we will forgo any such attempts.\\

The motivation for this theorem, which can be viewed as a generalization of Segre and Pl\"ucker embeddings, actually comes from our work on period morphisms of Rapoport-Zink spaces (cf. \cite{H7, HE}), where this embedding plays a role in understanding certain cycles in the cohomology of Rapoport-Zink towers coming from the Lubin-Tate tower (using the exterior powers of $p$-divisible groups defined in \cite{H2}).\\

This has also potential applications in intersection theory, Schubert calculus, arithmetic height functions and operadic calculus.\\

Let us now say a few words about the structure of the paper and the proof of the main theorem. We first start by developing some multilinear commutative algebra (section 2). In section 3, we ``translate'' the results of section 2 to the language of schemes and then glue these local results to obtain global multilinear algebra results on vector bundles over schemes. We then use these results, together with some techniques from algebraic geometry to show the main theorem.

\section{Some Multilinear Algebra}

Let $X$ be a scheme and $\CF$ a locally free $\CO_{X}$-module of  rank $n$. For $1\leq r\leq r$, we denote by $\ep^{r}\CF$ the sheaf $r^{\rm{th}}$-exterior power of $\CF$. It's a locally free $\CO_{X}$-module of rank $\binom{n}{r}$.

\begin{prop}
Let $R$ be a ring and $M$ and for $i=1,\dots,r$,  $M_{i}$ be elements of $\BM_{n}(R)$. Then, for every $1\leq d\leq n$ we have the following identities:
\begin{enumerate}[(1)]
\item $\det(M_{1}\otimes\dots\otimes M_{r})=\prod_{i=1}^{r}\big(\det(M_{i})^{\frac{\prod_{j=1}^{r} n_{j}}{n_{i}}}\big)$
\item $\det(\Sym^{d}M)=\det(M)^{\binom{n+d-1}{d-1}}$
\item $\det(\ep^{d}M)=\det(M)^{\binom{n-1}{d-1}}$
\end{enumerate}

\end{prop}

\begin{proof}
We will show the first identity; the other two can be proven similarly and therefore will be omitted.\\

If all $M_{i}$ are diagonal matrices, then a straightforward calculation shows the desired identity. This shows that the result also holds for diagonalizable matrices (this follows from the equality $\ep^{d}(A\cdot B)=\ep^{d}A\cdot\ep^{d}B$). Consider the map:

\[\Delta_{R}:\BM_{n_{1}}(R)\times\dots\times\BM_{n_{r}}(R)\to R\]
\[(M_{1},\dots,M_{r})\mapsto \det(M_{1}\otimes\dots\otimes M_{r})-\prod_{i=1}^{r}\big(\det(M_{i})^{\frac{\prod_{j=1}^{r} n_{j}}{n_{i}}}\big)\]
This is a continuous morphism with respect to the Zariski topology on both sides. Let us first assume that $R$ is an algebraically closed field. Then, the subset of $\BM_{n_{1}}(R)\times\dots\times\BM_{n_{r}}(R)$ consisting of diagonal matrices is dense. As the map is the constant zero map on diagonal matrices, it follows that it is identically zero on the whole space and we are done.\\

Now let $R$ be arbitrary. Write $M_{i}=(a_{ijk})_{j,k}$ and define the ring homomorphism:
\begin{align*}
\Theta:\BZ[X_{ijk};i=1,\dots,r]\to&\, R\\
X_{ijk}\mapsto&\, m_{ijk}
\end{align*}
Set $N_{i}:=(X_{ijk})_{j,k}\in \BM_{n_{i}}(\BZ[X_{ijk}])$. Under $\Theta$ these matrices map to $M_{i}$. The following diagram commutes:
\[\xymatrix{\BM_{n_{1}}(\BZ[X_{ijk}])\times\dots\times\BM_{n_{r}}(\BZ[X_{ijk}])\ar[rr]^{\Theta}\ar[d]_{\Delta_{\BZ[X_{ijk}]}}&&\BM_{n_{1}}(R)\times\dots\times\BM_{n_{r}}(R)\ar[d]^{\Delta_{R}}\\\BZ[X_{ijk}]\ar[rr]_{\Theta}&&R}\]Therefore, it is enough to show the identity (i) for $N_{i}$. This identity holds in $\BZ[X_{ijk}]$ if and only if it holds in an algebraic closure of the fraction field $\BQ(X_{ijk})$. But as we saw above, the identity holds in every algebraically closed field. This achieves the proof.
\end{proof}

\begin{cor}
\label{CorTenInj}
Let $R$ be a ring and $\phi_{i}:P_{i}\to Q_{i}$ ($i=1,\dots,r$) be $R$-linear homomorphisms between finitely generated projective $R$-modules of rank $n_{i}$. Then, the tensor product $\phi_{1}\otimes\dots\otimes\phi_{r}$ is injective (respectively surjective, respectively  an isomorphism) if and only if for all $i=1,\dots,r$, $\phi_{i}$ is injective (respectively surjective, respectively an isomorphism).
\end{cor}

\begin{proof}
Before we begin the proof, let us emphasize that, by Nakayama's lemma, a morphism between two finitely generated projective modules of the same rank is surjective if and only if it is an isomorphism and so, we included both in the statement just for aesthetics!\\

Note that the formation of tensor products commutes with base change and a homomorphism is injective (respectively an isomorphism) if and only if it is so after localization to each prime ideal. So, we can assume that $R$ is a local ring and therefore all finitely generated projective modules are free. So, we can assign a matrix $M_{i}\in\BM_{n_{i}}(R)$ to each $\phi_{i}$. Now, $\phi_{i}$ is injective (respectively an isomorphism) if and only if $\det(M_{i})$ is a non-zero-divisor (respectively a unit). The statement of the corollary now follows from the proposition.
\end{proof}

\begin{cor}
\label{CorAltInj}
Let $R$ be a ring and $\phi:P\to Q$ an $R$-linear homomorphism between two finitely generated projective $R$-modules of rank $n$. Then the following statements are equivalent:
\begin{itemize}
\item[(i)]
$\phi$ is injective (respectively surjective, respectively an isomorphism)
\item[(ii)]
$\ep^{d}\phi$ is injective (respectively surjective, respectively an isomorphism) for some $1\leq d\leq n$.
\item[(iii)]
$\ep^{d}\phi$ is injective (respectively surjective, respectively an isomorphism) for all $1\leq d\leq n$.
\end{itemize}
\end{cor}

\begin{proof}
The proof is similar to the proof of the previous corollary.
\end{proof}

\begin{cor}
\label{CorSymInj}
Let $R$ be a ring and $\phi:P\to Q$ an $R$-linear homomorphism between two finitely generated projective $R$-modules of the same rank. Then the following statements are equivalent:
\begin{itemize}
\item[(i)]
$\phi$ is injective (respectively surjective, respectively an isomorphism)
\item[(ii)]
$\Sym^{d}\phi$ is injective (respectively surjective, respectively an isomorphism) for some $d\geq 1$.
\item[(iii)]
$\Sym^{d}\phi$ is injective (respectively surjective, respectively an isomorphism) for all $d\geq 1$.
\end{itemize}
\end{cor}

\begin{proof}
The proof is similar to the proof of the Corollary \ref{CorTenInj}.
\end{proof}

\begin{lem}
\label{LemTensInj}
Let $R$ be a ring. Consider (for $i=1,\dots,r$) the following short exact sequences of finitely generated projective $R$-modules:
\[0\to V_{i}\arrover{\phi_{i}} W_{i}\to U_{i}\to 0\]
\[0\to V'_{i}\arrover{\phi_{i}'} W_{i}\to U'_{i}\to 0\]where  $\rank(V_{i})=\rank (V_{i}')$. Assume that the image of $\phi_{1}\otimes\dots\otimes\phi_{r}$ and the image of $\phi_{1}'\otimes\dots\otimes\phi_{r}'$ are equal as sub-$R$-modules of $W_{1}\otimes\dots\otimes W_{r}$. Then, for each $i$, the image of $\phi_{i}$ is the same as the image of $\phi_{i}'$ as sub-$R$-modules of $W_{i}$.
\end{lem}

\begin{proof}
By induction on $r$, we can assume $r=2$. Let us first assume that $R$ is local and therefore, all $V_{i}, V_{i}', W_{i}, U_{i}, U_{i}'$ are free. The two sequences in the lemma are split and so, we can complete a basis of $V_{i}$ to a basis of $W_{i}$ (here we identify $V_{i}$ and $V_{i}'$ with their image in $W_{i}$). Let $\{e_{1},e_{2},\dots,e_{n_{1}}\}$ and respectively $\{f_{1},f_{2},\dots,f_{n_{2}}\}$ be bases for $W_{1}$ and respectively $W_{2}$, with the first $m_{i}$ elements constituting a basis of $V_{i}$. Let $\epsilon_{i}=\sum_{j=1}^{n_{1}}a_{ji}e_{j}$ ($i=1,\dots,m_{1}$) and respectively $\delta_{\ell}=\sum_{k=1}^{n_{2}}b_{k\ell}f_{k}$ ($\ell=1,\dots,m_{2}$) be bases for $V_{1}'$ and respectively $V_{2}'$. If we show that for all $j>m_{1}$, all $k>m_{2}$ and all $i,\ell$ we have $a_{ji}=0=b_{k\ell}$, then we will have $V_{1}'\subseteq V_{1}$ and $V_{2}'\subseteq V_{2}$, and by symmetry also the opposite inclusions, which will finish the proof.\\

Let us write $A:=(a_{ji})\in\BM_{n_{1}\times m_{1}}(R)$ and $B:=(b_{k\ell})\in\BM_{n_{2}\times m_{2}}(R)$. Let us also write $\bar{A}$ and respectively $\bar{B}$ the square matrices obtained by deleting last rows of $A$ and respectively $B$. By assumption, the matrix $A\otimes B$ is of the following form:
\[\begin{pmatrix}
\bar{A}\otimes\bar{B}&\bf{0}\\\bf{0}&\bf{0}
\end{pmatrix}\]
We claim that the matrix $\bar{A}\otimes\bar{B}$ is invertible. Indeed, since  $\phi_{1}\otimes\dots\otimes\phi_{r}$ and $\phi_{1}'\otimes\dots\otimes\phi_{r}'$ are injective ($V_{i}, V_{i}'$ and $W_{i}$ are flat over $R$) and since they have the same image, matrix $\bar{A}\otimes\bar{B}$ is the invertible base change matrix from $\{e_{i}\otimes e_{j}\}$ to $\{\epsilon_{i}\otimes\epsilon_{j}\}$. It follows from Corollary \ref{CorTenInj} that $\bar{A}$ and $\bar{B}$ are invertible. It also follows that for all $j>m_{1}$ and all $i,k,\ell$, we have $a_{ji}b_{k\ell}=0$. So, we have $a_{ji}\bar{B}=0$. Since $\bar{B}$ is invertible, this implies that $a_{ji}=0$. Similarly, we have $b_{k\ell}=$ for all $k>m_{2}$ and all $\ell$. This is what we wanted and the proof is achieved in the case $R$ is a local ring.\\

Now assume that $R$ is an arbitrary ring. For any prime ideal $\Fp$, the sequences 
\[0\to V_{i,\Fp}\arrover{\phi_{i}} W_{i,\Fp}\to U_{i,\Fp}\to 0\]
\[0\to V'_{i,\Fp}\arrover{\phi_{i}'} W_{i,\Fp}\to U'_{i,\Fp}\to 0\] obtained by localization are exact. By the previous case, this means that the composition $V_{i,\Fp}\arrover{\phi_{i}} W_{i,\Fp}\to U_{i,\Fp}'$ is zero. Since this is true for all $\Fp$, we conclude that the composition $V_{i}\arrover{\phi_{i}}W_{i}\to U_{i}'$ is zero and so, $\phi_{i}$ factors through $\phi_{i}'$. The localization at $\Fp$ of the resulting morphism from $V_{i}$ to $V_{i}'$ is an isomorphism for all $\Fp$ and so this morphism is an isomorphism as well.
\end{proof}

\begin{lem}
\label{LemAltSymInj1}
Let $R$ be a ring and consider the following short exact sequences of finitely generated projective $R$-modules:

\[0\to V_{1}\arrover{\phi_{1}} W\to U_{1}\to 0\]
\[0\to V_{2}\arrover{\phi_{2}} W\to U_{2}\to 0\]
with $\rank V_{1}=\rank V_{2}=m$.\\

Assume that the image of $\ep^{r}\phi_{1}$ and $\ep^{r}\phi_{2}$ are equal as sub-$R$-modules of $\ep^{r}W$ (for some $1\leq r\leq m$). Then the image of $\phi_{1}$ and $\phi_{2}$ are equal as sub-$R$-modules of $W$.
\end{lem}

\begin{proof}
The proof of this lemma is similar to that of Lemma \ref{LemTensInj}, but requires a slight modification. As in the proof of Lemma \ref{LemTensInj}, it is enough to consider the case where $R$ is local ring, and so, $W, V_{i}, U_{i}$ ($i=1,2$) are free. Let $\{e_{1}, e_{2},\dots,e_{n}\}$ be a basis of $W$ with the first $m$ elements forming a basis for $V_{1}$. Let $\epsilon_{i}=\sum_{j=1}^{n}a_{ji}e_{j}$ ($i=1,\dots, m$) be a basis for $V_{2}$. Again, we want to show that for all $j>m$ and all $i$ the coefficient $a_{ji}$ vanishes. Write $A:=(a_{ji})\in\BM_{n\times m}(R)$. Let $\bar{A}$ be the $m\times m$ matrix obtained by deleting the last $n-m$ rows of $A$. By assumption the matrix $\ep^{r}A$ is of the form:
\[\begin{pmatrix}\ep^{r}\bar{A}&\bf{0}\\\bf{0}&\bf{0}\end{pmatrix}\]
and the matrix $\ep^{r}\bar{A}$ is invertible. By Corollary \ref{CorAltInj}, $\bar{A}$ is invertible as well.\\

For all $(i_{1},\dots,i_{r})$ with $1\leq i_{1}<\dots<i_{r}\leq n$ and at least one $i_{k}>m$, the $(i_{1},\dots,i_{r})$-row of $\ep^{r}A$ is zero. This means that all $r$-minors of $A$, where we keep at least one of the last $n-m$ rows, are zero. This implies the same statement for all $r+1$-minors, where we keep at least one of the last $n-m$ rows. By induction, this means that all $m$-minors of $A$, where we keep at least one of the last $n-m$ rows are zero. We also know that $\bar{A}$ is invertible. Let us write $A_{i}$ for the $i^{\rm{th}}$-row of $A$. Fix a $j>m$ and write $\check{A}_{j}$ for the matrix 
\[\begin{pmatrix}\bar{A}\\A_{j}\end{pmatrix}\]
in other words, add row $A_{j}$ to the bottom of $\bar{A}$. Then all the $m$-minors of $\check{A}_{j}$ are zero if we keep the last row. Noting that the determinant is multilinear in rows, this will remain the case if we multiply $\check{A}_{j}$ with the matrix 
\[\begin{pmatrix}(\bar{A})^{-1}&\bf{0}\\\bf{0}&1\end{pmatrix}
\]
We have 
\[\begin{pmatrix}(\bar{A})^{-1}&\bf{0}\\\bf{0}&1\end{pmatrix}\begin{pmatrix}\bar{A}\\A_{j}\end{pmatrix}=\begin{pmatrix}{I}_{m}\\A_{j}\end{pmatrix}\]
here ${I}_{m}$ is the identity matrix of size $m$. This implies that $A_{j}$ is zero. This is what we wanted and the proof is achieved.
\end{proof}

\begin{lem}
\label{LemAltSymInj2}
Let $m$ and $r$ be natural numbers and let $R$ be a $\BZ[r^{m-2}]$-algebra. Consider the following short exact sequences of finitely generated projective $R$-modules:

\[0\to V_{1}\arrover{\phi_{1}} W\to U_{1}\to 0\]
\[0\to V_{2}\arrover{\phi_{2}} W\to U_{2}\to 0\]
with $\rank V_{1}=\rank V_{2}=m$.\\

Assume that the image of $\Sym^{r}\phi_{1}$ and $\Sym^{r}\phi_{2}$ are equal as sub-$R$-modules of $\Sym^{r}W$. Then the image of $\phi_{1}$ and $\phi_{2}$ are equal as sub-$R$-modules of $W$.
\end{lem}

\begin{proof}
We proceed as in the proof of the previous lemma and keep the same notations (e.g., $\{e_{1},\dots,e_{n}\}$ is a basis of $W$, the first $m$ elements form a basis of $V_{1}$, $\bar{A}$ is the square $m\times m$-matrix obtained from $A$, etc.). Again, we have to show that for all $j>m$ and all $i$, the coefficient $a_{ji}$ is zero. By symmetry, we can assume that $j=m+1$ and $i=1$.\\

The condition $R$ be a $\BZ[r^{m-2}]$-algebra means 
that either $m\geq 2$ or that $r$ is invertible in $R$. Let us first assume that $m\geq 2$.\\

By assumption, we have

\begin{align}
\label{EqSymrA}
\Sym^{r}A=\begin{pmatrix}\Sym^{r}\bar{A}&\bf{0}\\\bf{0}&\bf{0}\end{pmatrix}
\end{align} 

and the matrix $\Sym^{r}\bar{A}$ is invertible, and so by Corollary \ref{CorSymInj}, $\bar{A}$ is invertible as well.\\

Given $r$ elements $w_{1},\dots,w_{r}\in W$, we write $w_{1}w_{2}\dots w_{r}$ for their ``product'' in $\Sym^{r}W$. Fix a vector $I=(i_{1},\dots,i_{m})$ of non-negative integers with $i_{1}+\dots+i_{m}=r-1$. The coefficient of the basis element $e_{1}^{i_{1}}\dots e_{m}^{i_{m}}e_{m+1}$ in the expansion of the element $$\epsilon_{1}\epsilon_{2}^{r-1}=(a_{1,1}e_{1}+\dots+a_{n,1}e_{n})(a_{1,2}e_{1}+\dots+a_{n,2}e_{n})^{r-1}\in \Sym^{r}W$$ is:

\begin{align}
\label{EqCoeff1&2}
a_{m+1,1}\binom{r-1}{i_{1},\dots,i_{m}}\prod_{k=1}^{m}a_{k,2}^{i_{k}} &+\nonumber \\ \sum_{j=1}^{m}a_{m+1,2}a_{j,1}\frac{\prod_{k=1}^{m}a_{k,2}^{i_{k}}}{a_{j,2}}\binom{r-1}{i_{1},\dots,i_{j-1},i_{j}-1,i_{j+1},\dots,i_{m},1}\nonumber&=\\
a_{m+1,1}\binom{r-1}{i_{1},\dots,i_{m}}\prod_{k=1}^{m}a_{k,2}^{i_{k}}+a_{m+1,2}(r-1)\sigma_{I}
\end{align}

where we denote by $\binom{M}{k_{1},\dots,k_{\ell}}$ the generalized binomial coefficient $\frac{M!}{k_{1}!\dots k_{\ell}!}$, with the convention that if one of $k_{j}$ is negative, then the coefficient $\binom{M}{k_{1},\dots,k_{\ell}}$ is zero, and where we denote $$\sigma_{I}:=\sum_{j=1}^{m}a_{j,1}\frac{\prod_{k=1}^{m}a_{k,2}^{i_{k}}}{a_{j,2}}\binom{r-2}{i_{1},\dots,i_{j-1},i_{j}-1,i_{j+1},\dots,i_{m}}$$

By (\ref{EqSymrA}), this coefficient is zero, i.e.,

\begin{align}
\label{EqCoeffZero}
a_{m+1,1}\binom{r-1}{i_{1},\dots,i_{m}}\prod_{k=1}^{m}a_{k,2}^{i_{k}}+a_{m+1,2}(r-1)\sigma_{I}=0
\end{align}

Now, let us look at the entries of the matrix $\Sym^{r-1}\bar{A}$ at row $I$ and columns $I_{1}:=(0,r-1,0,\dots,0)$ and $I_{2}:=(1,r-2,0,\dots,0)$, denoted respectively by $s_{1}$ and $s_{2}$. For $i=1,2$, $s_{i}$ is the coefficient of $e_{1}^{i_{1}}\dots e_{m}^{i_{m}}$ in the vector $$\Sym^{r-1}\bar{A}\,(e_{1}^{i-1}e_{2}^{r-i})\in\Sym^{r-1}V_{1}$$
We therefore have:

\begin{align}
\label{S1&2}
s_{1}&=\binom{r-1}{i_{1},\dots,i_{m}}\prod_{k=1}^{m}a_{k,2}^{i_{k}}\nonumber\\
s_{2}&=\sigma_{I}
\end{align}

Set $\nu\in\Sym^{r-1}V_{1}$ to be the vector whose component at row $I_{1}$ and respectively $I_{2}$ is $a_{m+1,1}$ and respectively $a_{m+1,2}(r-1)$, and zeros everywhere else. Note that $\nu$ is independent from $I$. It follows from (\ref{EqCoeffZero}) and (\ref{S1&2}) that $$\Sym^{r-1}\bar{A}\cdot \nu=0$$

But as $\bar{A}$ is invertible, $\Sym^{r-1}\bar{A}$ is invertible as well (see Corollary \ref{CorSymInj}) and so $\nu=0$, which implies that $a_{m+1,1}=0$, as desired.\\

Now, let us assume that $m=1$ and $r$ is invertible in $R$. To simplify the notation, let us write $\epsilon=a_{1}e_{1}+\dots a_{n}e_{n}$ for a basis of $V_{2}$. Then, by assumption, $\epsilon^{r}$ is a basis of $\Sym^{r}V_{1}$. In other words, $\epsilon^{r}$ is a unit times $e_{1}^{r}$. We have

\[\epsilon^{r}=\sum_{\substack{0\leq i_{1},\dots,i_{n}\\i_{1}+\dots+i_{n}=r}}\binom{r}{i_{1},\dots,i_{n}}a_{1}^{i_{1}}\dots a_{n}^{i_{n}}e_{1}^{i_{1}}\dots e_{n}^{i_{n}}=\]
\[a_{1}^{r}e_{1}^{r}+\sum_{i=2}^{n}ra_{1}^{r-1}a_{i}e_{1}^{r-1}e_{i}+\sum_{\substack{0\leq i_{1},\dots,i_{n}\\ i_{1}<r-1\\i_{1}+\dots+i_{n}=r}}\binom{r}{i_{1},\dots,i_{n}}a_{1}^{i_{1}}\dots a_{n}^{i_{n}}e_{1}^{i_{1}}\dots e_{n}^{i_{n}}\]

It follows that $a_{1}^{r}$ is a unit and for all $i=2,\dots,n$, the coefficient $ra_{1}^{r-1}a_{i}$ is zero. Since $r$ is a unit in $R$, we conclude that for all $i=2,\dots,n$ we have $a_{i}=0$ as desired. This finishes the proof.
\end{proof}

\begin{rem}
We will need the last three lemmas only in the case where $R$ is a local ring, however, since the proof for arbitrary ring did not need much more work and its statement might be interesting in itself, we decided to state it more generally.\xqed{\lozenge}
\end{rem}

\section{Main Theorem}

Let us now sheafify Lemmas \ref{LemTensInj}, \ref{LemAltSymInj1} and \ref{LemAltSymInj2} to get similar results for short exact sequences of vector bundles over schemes.

\begin{prop}
\label{PropTensInj}
Let $X$ be a scheme. Consider (for $i=1,\dots,r$) the following short exact sequences of finite locally free  $\CO_{X}$-modules:
\[0\to \CF_{i}\arrover{\phi_{i}} \CG_{i}\to \CH_{i}\to 0\]
\[0\to \CF'_{i}\arrover{\phi_{i}'} \CG_{i}\to \CH'_{i}\to 0\]where  $\rank(\CF_{i})=\rank (\CF_{i}')$. Then, we have a factorization 

\begin{align}
\label{FactPsi1}
\xymatrix{\CF_{1}\otimes_{\CO_{X}}\dots\otimes_{\CO_{X}}\CF_{r}\ar[rr]^{\phi_{1}\otimes\dots\otimes\phi_{r}}\ar@{-->}[dr]_{\Psi}^{\cong}&&\CG_{1}\otimes_{\CO_{X}}\dots\otimes_{\CO_{X}}\CG_{r}\\
&\CF_{1}'\otimes_{\CO_{X}}\dots\otimes_{\CO_{X}}\CF_{r}'\ar[ur]_{\phi_{1}'\otimes\dots\otimes\phi_{r}'}&}
\end{align}

where $\Psi$ is an isomorphism if and only if for all $i=1,\dots,r$ we have a factorization 

\begin{align}
\label{FactPsi2}
\xymatrix{\CF_{i}\ar[rr]^{\phi_{i}}\ar@{-->}[dr]_{\Psi_{i}}^{\cong}&&\CG_{i}\\
&\CF_{i}'\ar[ur]_{\phi_{i}'}&}
\end{align}

where $\Psi_{i}$ is an isomorphism.
\end{prop}

\begin{proof}
If we have isomorphisms $\Psi_{i}$ as in (\ref{FactPsi2}) then the tensor product $\Psi_{1}\otimes\dots\otimes\Psi_{r}$ is an isomorphism satisfying (\ref{FactPsi1}). So, let's assume that we have an isomorphism as in (\ref{FactPsi1}).  By induction on $r$, we can assume that $r=2$. For any $x\in X$, the sequences 
\[0\to \CF_{i,x}\arrover{\phi_{i,x}} \CG_{i,x}\to \CH_{i,x}\to 0\]
\[0\to \CF'_{i,x}\arrover{\phi_{i,x}'} \CG_{i,x}\to \CH'_{i,x}\to 0\] are exact sequences of finite free $\CO_{X,x}$-modules. By previous lemma, this means that the composition $\CF_{i,x}\arrover{\phi_{i,x}} \CG_{i,x}\to \CH_{i,x}'$ is zero. Since this is true for all $x\in X$, we conclude that the composition $\CF_{i}\arrover{\phi_{i}}\CG_{i}\to \CH_{i}'$ is zero and so, $\phi_{i}$ factors through $\phi_{i}'$. The stalk at $x$ of the resulting morphism $\Psi_{i}:\CF_{i}\to\CF_{i}'$ is an isomorphism for all $x$ and so $\Psi_{i}$ is an isomorphism as well, making the diagram (\ref{FactPsi2}) commutative.
\end{proof}

\begin{prop}
\label{PropAltSymInj1}
Let $X$ be a scheme. Consider the following short exact sequences of finite locally free  $\CO_{X}$-modules:
\[0\to \CF_{1}\arrover{\phi_{1}} \CG\to \CH_{1}\to 0\]
\[0\to \CF_{2}\arrover{\phi_{2}} \CG\to \CH_{2}\to 0\]where  $\rank(\CF_{1})=\rank (\CF_{2})$.\\

We have a factorization (for some $1\leq r\leq m$):

\begin{align}
\label{FactPsi1}
\xymatrix{\ep^{r}_{\CO_{X}}\CF_{1}\ar[rr]^{\ep^{r}\phi_{1}}\ar@{-->}[dr]_{\Psi}^{\cong}&&\ep^{r}_{\CO_{X}}\CG\\
&\ep^{r}_{\CO_{X}}\CF_{2}\ar[ur]_{\ep^{r}\phi_{2}}&}
\end{align}

where $\Psi$ is an isomorphism if and only if we have a factorization 

\begin{align}
\label{FactPsi2}
\xymatrix{\CF_{1}\ar[rr]^{\phi_{1}}\ar@{-->}[dr]_{\tilde{\Psi}}^{\cong}&&\CG\\
&\CF_{2}\ar[ur]_{\phi_{2}}&}
\end{align}

where $\tilde{\Psi}$ is an isomorphism.
\end{prop}

\begin{proof}
The proof is similar to the proof of Proposition \ref{PropTensInj}, based on Lemma \ref{LemAltSymInj1}, and is therefore omitted.
\end{proof}

\begin{prop}
\label{PropAltSymInj2}
Let $m$ and $r$ be natural numbers. Let $X$ be a $\BZ[r^{m-2}]$-scheme. Consider the following short exact sequences of finite locally free  $\CO_{X}$-modules:
\[0\to \CF_{1}\arrover{\phi_{1}} \CG\to \CH_{1}\to 0\]
\[0\to \CF_{2}\arrover{\phi_{2}} \CG\to \CH_{2}\to 0\]where  $\rank(\CF_{1})=m=\rank (\CF_{2})$.\\

We have a factorization

\begin{align}
\label{FactPsi3}
\xymatrix{\Sym^{r}_{\CO_{X}}\CF_{1}\ar[rr]^{\Sym^{r}\phi_{1}}\ar@{-->}[dr]_{\Psi}^{\cong}&&\Sym^{r}_{\CO_{X}}\CG\\
&\Sym^{r}_{\CO_{X}}\CF_{2}\ar[ur]_{\Sym^{r}\phi_{2}}&}
\end{align}

where $\Psi$ is an isomorphism if and only if we have a factorization 

\begin{align}
\label{FactPsi4}
\xymatrix{\CF_{1}\ar[rr]^{\phi_{1}}\ar@{-->}[dr]_{\tilde{\Psi}}^{\cong}&&\CG\\
&\CF_{2}\ar[ur]_{\phi_{2}}&}
\end{align}

where $\tilde{\Psi}$ is an isomorphism.
\end{prop}

\begin{proof}
The proof is similar to the proof of Proposition \ref{PropTensInj}, based on Lemma \ref{LemAltSymInj2}, and is therefore omitted.
\end{proof}

\begin{notation}
Let $X$ be a scheme. Let $\CV$ be a finite locally free $\CO_{X}$-module. We denote by $\BG r(\CV,d)$ the Grassmannian (flag) variety of short exact sequences of locally free sheaves \[0\to \CF\to\CV\to\CG\to 0\] where $\CF$ is of rank $d$. In other words, $\BG r(\CV,d)$ represents the functor from $X$-schemes to sets, where any $X$-scheme $S$ is mapped to the set of short exact sequences 
\[0\to \CF\to\CV_{S}\to\CG\to 0\]of finite locally free $\CO_{S}$-modules with $\CF$ of rank $d$. Two such sequences 
\[0\to \CF\to\CV_{S}\to\CG\to 0\]
and \[0\to \CF'\to\CV_{S}\to\CG'\to 0\]
are identified, if there is a commutative diagram:
\[\xymatrix{0\ar[r]& \CF\ar[r]\ar[d]_{\cong}&\CV_{S}\ar[r]\ar@{=}[d]&\CG\ar[r]\ar[d]& 0\\
0\ar[r]& \CF'\ar[r]&\CV_{S}\ar[r]&\CG'\ar[r]& 0}\]Note that the morphism $\CF\to\CF'$ should be an isomorphism and the morphism $\CV_{S}\to\CV_{S}$ is the identity. The morphism $\CG\to\CG'$ will then automatically be an isomorphism. We could have required the morphism $\CG\to\CG'$ be an isomorphism, which would then imply that $\CF\to\CF'$ is an isomorphism as well (for details about Grassmannians see \cite[\href{https://stacks.math.columbia.edu/tag/089R}{Section 089R}]{stacks-project} or \cite{MR3075000} Ch.I, \S9.7).
\end{notation}

\begin{cons}
Let $X$ be a scheme. Let $\CV_{i}$ (for $i=1,\dots,r$) be a finite locally free $\CO_{X}$-module and fix natural numbers $m_{i}\leq \rank(\CV_{i})$. We want to construct a morphism of $X$-schemes $$\FT:\BG r(\CV_{1},m_{1})\times_{X}\dots\times_{X}\BG r(\CV_{r},m_{r})\to \BG r(\CV_{1}\otimes_{\CO_{X}}\dots\otimes_{\CO_{X}}\CV_{r},m_{1}\cdots m_{r})$$It is enough to construct, functorially in $X$-schemes $S$, a map of sets \[\FT:\BG r(\CV_{1},m_{1})(S)\times\dots\times\BG r(\CV_{r},m_{r})(S)\to \BG r(\CV_{1}\otimes_{\CO_{X}}\dots\otimes_{\CO_{X}}\CV_{r},m_{1}\cdots m_{r})(S)\]We are given $r$ short exact sequences \[(\xi_{i})\qquad 0\to\CF_{i}\to\CV_{i,S}\to\CG_{i}\to0\] of finite locally free $\CO_{S}$-modules with $\CF_{i}$ of rank $m_{i}$. Define $\FT(\xi_{1},\dots,\xi_{r})$ to be the following short exact sequence:

\[0\to \CF_{1}\otimes_{\CO_{S}}\dots\otimes_{\CO_{S}}\CF_{r}\to \CV_{S,1}\otimes_{\CO_{S}}\dots\otimes_{\CO_{S}}\CV_{S,r}\to\CC\to0\]where $\CC$ is the cokernel of $\CF_{1}\otimes_{\CO_{S}}\dots\otimes_{\CO_{S}}\CF_{r}\to \CV_{S,1}\otimes_{\CO_{S}}\dots\otimes_{\CO_{S}}\CV_{S,r}$. It is not hard to see that $\CC$ is indeed locally free. For, if we choose an affine covering of $S$, then the sequences $\xi_{i}$ split and so, over this affine covering, $\CC$ is a direct sum of tensor products of finite locally free sheaves. This shows that $\CC$ is also locally free. The construction of $\FT$ is functorial is $S$ and so defines the desired morphism.\xqed{\blacktriangledown}
\end{cons}

\begin{cons}
In a similar fashion, if $X$ is a scheme, $\CV$ is a finite locally free $\CO_{X}$-module and $m\leq\rank(\CV)$ is a natural number, then we can construct two morphisms of $X$-schemes:
\[\FA_{r}:\BG r(\CV,m)\to \BG r\Big(\ep^{r}_{\CO_{X}}\CV,\binom{m}{r}\Big)\] (for $1\leq r\leq m$) and respectively 
\[\FS_{r}:\BG r(\CV,m)\to \BG r\Big(\Sym^{r}_{\CO_{X}}\CV,\binom{m+r-1}{r}\Big)\]
(for $r\geq 1$) using the exterior and respectively the symmetric power operators.\\

Composing $\FT$ with the diagonal morphism $$\Delta_{r}:\BG r(\CV,m)\into \BG r(\CV,m)\times_{X}\dots\times_{X}\BG r(\CV,m)$$ we also obtain a third morphism of $X$-schemes:
\[\FT_{r}:\BG r(\CV,m)\to \BG r\big(\CV^{\otimes r},m^{r}\big)\]
(for $r\geq 1$).\xqed{\blacktriangledown}
\end{cons}

\begin{prop}
Let $X$ be a scheme. Let $\CV_{i}$ (for $i=1,\dots,r$) and $\CV$ be finite locally free $\CO_{X}$-modules and fix natural numbers $m_{i}\leq \rank(\CV_{i})$ and $m\leq\rank(\CV)$. Then, for any $X$-scheme $S$, maps $\FT(S), \FT_{r}(S)$ and $\FA_{r}(S)$ are injective, in other words, $\FT, \FT_{r}$ and $\FA$ are monomorphisms in the category of schemes.
\end{prop}

\begin{proof}
The statements for $\FT$ and  $\FA_{r}$ are direct consequences of Propositions \ref{PropTensInj} and \ref{PropAltSymInj1}. The statement for $\FT_{r}$ follows from the fact that Grassmannians are separated and therefore the diagonal morphism $\Delta_{r}$ is a closed immersion (immersions are mononorphism \cite[\href{https://stacks.math.columbia.edu/tag/01L7}{Lemma 01L7}]{stacks-project}). %second part follows from the fact that a morphism of schemes is universally injective if and only if for any field $K$ the induced map on $K$-points is injective (cf. \cite[\href{https://stacks.math.columbia.edu/tag/01S4}{Lemma 01S4}]{stacks-project}).
\end{proof}

\begin{prop}
Let $m$ and $r$ be natural numbers and let $X$ be a $\BZ[r^{m-2}]$-scheme. Let $\CV$ be a finite locally free $\CO_{X}$-module and assume $m\leq\rank(\CV)$. Then, for any $X$-scheme $S$, $\FS_{r}(S)$ is injective, in other words, $\FS$ is a monomorphism in the category of schemes.
\end{prop}

\begin{proof}
This is a direct consequence of Proposition \ref{PropAltSymInj2}.
\end{proof}

%\begin{lem}
%Let $T$ and $S$ be schemes and $f:T\to S$ a propre morphism of finite presentation. Assume that for all rings $R$, the induced map $T(R)\to S(R)$ is injective. Then $f$ is a closed immersion.
%\end{lem}

%\begin{proof}
%Let us first show that $f$ is a monomorphism in the category of schemes. By assumption, for any affine scheme $U$, the induced morphism \[\Hom(U,T)\to\Hom(U,S)\] is injective. Now, let $W$ be a scheme and let $g, h:W\to T$ be two morphisms, whose compositions with $f$ are equal. Cover $W$ with affine schemes $U_{\alpha}$ ($\alpha\in\Lambda$). Then, the compositions with $f$ of the restrictions $g_{|_{U_{\alpha}}}$ and $h_{|_{U_{\alpha}}}$ are equal. By assumption, we then have $g_{|_{U_{\alpha}}}=h_{|_{U_{\alpha}}}$. Since $U_{\alpha}$ cover $W$, we conclude that $g=h$. As $f$ is a monomorphism, is proper and is of finite presentation, by Proposition 8.11.5 of \cite{MR0217086}, it is a closed immersion. 
%\end{proof}

\begin{thm}
\label{MainThm}
Let $X$ be a scheme. Let $\CV_{i}$ (for $i=1,\dots,r$) and $\CV$ be finite locally free $\CO_{X}$-module and fix natural numbers $m_{i}\leq \rank(\CV_{i})$ and $m\leq \rank(\CV)$. The $X$-morphisms  

\begin{align*}
\FT:\BG r(\CV_{1},m_{1})\times_{X}\dots\times_{X}\BG r(\CV_{r},m_{r})&\to \BG r(\CV_{1}\otimes_{\CO_{X}}\dots\otimes_{\CO_{X}}\CV_{r},m_{1}\cdots m_{r})\\
\FT_{r}:\BG r(\CV,m)&\to \BG r\big(\CV^{\otimes r},m^{r}\big)\\
\FA_{r}:\BG r(\CV,m)&\to \BG r\Big(\ep^{r}_{\CO_{X}}\CV,\binom{m}{r}\Big)\\
\end{align*}

are closed immersions. If $m>1$ or $r$ is invertible on $X$, then 
\[\FS_{r}:\BG r(\CV,m)\to \BG r\Big(\Sym^{r}_{\CO_{X}}\CV,\binom{m+r-1}{r}\Big)\]
is a closed immersion as well.
\end{thm}

\begin{proof}
We will only prove that $\FT$ is a closed immersion. The other cases are similar. In order to simplify the notations, let us write $Y:=\BG r(\CV_{1},m_{1})\times_{X}\dots\times_{X}\BG r(\CV_{r},m_{r})$ and $Z:=\BG r(\CV_{1}\otimes_{\CO_{X}}\dots\otimes_{\CO_{X}}\CV_{r},m_{1}\cdots m_{r})$. By \cite{MR3075000} Proposition 9.8.4, Grassmannians are closed subschemes of projective spaces over $X$ (via Pl\"ucker embedding) and therefore are proper over $X$. It follows that $Y$ and $Z$ are proper over $X$. Since $\FT$ is a morphism between proper schemes, it is proper as well (cf. \cite{MR0163909}, \S5, Corollaire 5.4.3.).\\

If we can show that $\FT$ is of finite presentation (cf. \cite{MR0173675} D\'efinition 1.6.1. or \cite[\href{https://stacks.math.columbia.edu/tag/02FV}{Lemma 02FV}]{stacks-project}) then, since $\FT$ is proper and is a monomorphism by previous proposition, it follows from \cite{MR0217086} Proposition 8.11.5. that it is a closed immersion.\\

Since $\CV_{i}$ are finite locally free, by \cite{MR3075000} Proposition 9.8.4., $Y$ and $Z$ embed as a closed subscheme with a coherent ideal sheaf into a product of projective schemes over $X$. Since projective schemes are of finite presentation, it follows that $Y$ and $Z$ are of finite presentation over $X$. Since $Z$ is proper and therefore quasi-separated over $X$, it follows from \cite{MR0173675} Proposition 1.6.2. (v) that $\FT$ is of finite presentation.%Since by previous proposition it is injective, it follows that $\FT$ is a homeomorphism onto its image, which is a closed subset. The injectivity also implies that $\FT$ is quasi-finite. Being proper and quasi-finite $\FT$ is therefore finite (cf. \cite[\href{https://stacks.math.columbia.edu/tag/02LS}{Lemma 02LS}]{stacks-project} or \cite{MR0217086} Th\'eor\`eme 8.11.1.).
\end{proof}

\begin{rem}
For the symmetric power to be a closed immersion, the condition that $m$ is at least $2$ or $r$ is invertible on $X$ is actually necessary (unless the rank of the ambient space $\CV$ is $1$, in which case, we have the identity of the scheme $X$!). In order to see this, assume that $m=1$ and that $r$ is not a unit of $\Gamma(X,\CO)$. Then, there is a local Noetherian ring $R$ of characteristic $p>0$ with $p | r$, an element $\alpha\in R\setminus\{0\}$ with $\alpha^{2}=0$, and a morphism $\Spec(R)\to X$. We are going to show that the map $\FS_{r}(\Spec(R))$ is not injective, and so, $\FS_{r}$ cannot be a closed immersion.\\

Let $W$ be the free $R$-module of rank say $n$ associated with $\CV$. Choose a basis $e_{1},\dots,e_{n}$ of $W$ and let $V_{1}$ be the submodule generated by $e_{1}$. Then $V_{1}\subset W$ defines a point on $\BG r(W,1)$ on $R$. Let $V_{2}$ be the submodule generated by $e_{1}+\alpha e_{2}$. It is straightforward to see that $V_{2}\subset W$ defines another point of $\BG r(W,1)$ over $R$ that is distinct from $V_{1}\subset W$. However, we have 
\[(e_{1}+\alpha e_{2})^{r}=\big((e_{1}+\alpha e_{2})^{p}\big)^{r/p}=(e_{1}^{p}+\alpha^{p}e_{2}^{p})^{r/p}=e_{1}^{r}\] and so $\Sym^{r}V_{2}=\Sym^{r}V_{1}$.
\xqed{\lozenge}
\end{rem}

\begin{rem}
Note that the embedding $\FT$ generalizes the Segre embedding (cf. \cite{MR0163909}, \S4.3)
% \[\BP^{n}_{X}\times\BP^{m}_{X}\into \BP^{nm+n+m}\] or in a more elegant way the embedding
\[\BP(\CE)\times_{X}\BP(\CF)\into\BP(\CE\otimes_{\CO_{X}}\CF)\]where $\CE$ and $\CF$ are finite locally free $\CO_{X}$-modules and the embedding $\FA_{r}$ generalizes the Pl\"ucker embedding (cf. \cite{MR3075000}, \S9.8)
\[\BG r(\CV,m)\into \BP(\ep^{m}_{\CO_{X}}\CV)\] \xqed{\lozenge}
\end{rem}

\begin{rem}
It is possible to generalize these results to more general flag varieties. These general results can be obtained by induction from what we already have. There are many ways of constructing flags from a given flag and using various tensor constructions. A general statement would be rather nasty and we will therefore avoid it.\xqed{\lozenge}
\end{rem}

\textbf{Acknowledgments.} I would like to thank Arash Rastegar and Simon H\"aberli for helpful conversations.
\newpage

\phantomsection
\addcontentsline{toc}{section}{References}

\footnotesize{

\bibliographystyle{acm}}

\end{document}